\documentclass[english,12pt]{article}
\usepackage[T1]{fontenc}
\usepackage[latin9]{inputenc}
\usepackage{geometry}
\usepackage{bm}
\usepackage{amsmath}
\usepackage{amsthm}
\usepackage{amssymb}
\usepackage{setspace}
\usepackage[authoryear]{natbib}
\makeatletter
\newcommand{\lyxaddress}[1]{
	\par {\raggedright #1
	\vspace{1.4em}
	\noindent\par}
}
\theoremstyle{plain}
\newtheorem{prop}{\protect\propositionname}
\theoremstyle{plain}
\newtheorem{lem}{\protect\lemmaname}

\makeatother

\usepackage{babel}
\providecommand{\lemmaname}{Lemma}
\providecommand{\propositionname}{Proposition}

\begin{document}
\title{Universal Inference with Composite Likelihoods}
\author{Hien Duy Nguyen\thanks{Email: h.nguyen5@latrobe.edu.au.}}
\maketitle

\lyxaddress{Department of Mathematics and Statistics, La Trobe University, Bundoora
3086, Australia}
\begin{abstract}
\citet[PNAS, vol. 117, pp. 16880\textendash 16890]{Wasserman:2020aa}
constructed estimator agnostic and finite-sample valid confidence
sets and hypothesis tests, using split-data likelihood ratio-based
statistics. We demonstrate that the same approach extends to the use
of split-data composite likelihood ratios as well, and thus establish
universal methods for conducting multivariate inference when the data
generating process is only known up to marginal and conditional relationships
between the coordinates. Always-valid sequential inference is also
considered.
\end{abstract}

\section{Introduction}

Let $\bm{X}\in\mathbb{X}\subseteq\mathbb{R}^{d}$ ($d\in\mathbb{N}$)
be a random variable arising from a parametric family of distributions
$\mathcal{P}_{\bm{\theta}}$ with probability density/mass functions
(we shall use PDFs/PMFs) of form $p\left(\bm{x};\bm{\theta}\right)$,
for $\bm{\theta}\in\Theta\subseteq\mathbb{R}^{q}$ ($q\in\mathbb{N}$).
Let $\mathbf{X}_{2n}=\left(\bm{X}_{1},\dots,\bm{X}_{2n}\right)$ be
a sample of $2n$ ($n\in\mathbb{N}$) independently and identically
distributed replicates of $\bm{X}$ and split the data into two subsamples
$\mathbf{X}_{n}^{0}=\left(\bm{X}_{1},\dots,\bm{X}_{n}\right)=\left(\bm{X}_{1}^{0},\dots,\bm{X}_{n}^{0}\right)$
and $\mathbf{X}_{n}^{1}=\left(\bm{X}_{n+1},\dots,\bm{X}_{2n}\right)=\left(\bm{X}_{1}^{1},\dots,\bm{X}_{n}^{1}\right)$.
Without causing confusion, we shall use PDF to mean PDF or PMF, throughout
the text.

Suppose that the data generating process (DGP) of $\bm{X}$ has distribution
$\mathcal{P}_{\bm{\theta}^{*}}$ for some $\bm{\theta}^{*}\in\Theta$
and that $\tilde{\bm{\theta}}_{n}^{k}$ is some generic estimator
of $\bm{\theta}^{*}$, using data $\mathbf{X}_{n}^{k}$ ($k\in\left\{ 0,1\right\} $).
Consider the split likelihood ratio statistics (LRSs)
\begin{equation}
U_{n}^{k}\left(\bm{\theta}\right)=\frac{L\left(\tilde{\bm{\theta}}_{n}^{1-k};\mathbf{X}_{n}^{k}\right)}{L\left(\bm{\theta};\mathbf{X}_{n}^{k}\right)}\text{,}\label{eq: T like}
\end{equation}
for each $k$, and the swapped LRS

\begin{equation}
\bar{U}_{n}\left(\bm{\theta}\right)=\frac{U_{n}^{0}\left(\bm{\theta}\right)+U_{n}^{1}\left(\bm{\theta}\right)}{2}\text{,}\label{eq: S like}
\end{equation}
and 
\[
L\left(\bm{\theta};\mathbf{X}_{n}^{k}\right)=\prod_{i=1}^{n}p\left(\bm{X}_{i}^{k};\bm{\theta}\right)
\]
is the likelihood of subsample $\mathbf{X}_{n}^{k}$, evaluated at
parameter value $\bm{\theta}$.

Let $\text{E}_{\bm{\theta}^{*}}$ and $\mathrm{Pr}_{\bm{\theta}^{*}}$
denote the expectation and probability operators with respect to the
distribution $\mathcal{P}_{\bm{\theta}^{*}}$, respectively. In \citet{Wasserman:2020aa},
the remarkable result that 

\begin{equation}
\text{E}_{\bm{\theta}^{*}}\left[U_{n}^{k}\left(\bm{\theta}^{*}\right)\right]\le1\label{eq: Wasserman Main}
\end{equation}
is established and used to derive finite-sample validity of a number
of simple universal confidence set estimators and hypothesis tests,
using (\ref{eq: T like}) and (\ref{eq: S like}) (and variants),
that are agnostic to the choice of parameter estimators $\tilde{\bm{\theta}}_{n}^{k}$
and DGPs $\mathcal{P}_{\bm{\theta}}$. The results are then extended
from likelihood-based inference to misspecified likelihood, power
likelihood, and smoothed likelihood-based inference, as per the works
of \citet{White1982}, \citet{Royall:2003aa}, and \citet{Seo:2013aa},
respectively. Furthermore, \citet{Wasserman:2020aa} prove results
regarding always-valid tests, $p$-values and confidence sets, in
the style of \citet{Johari:2017aa}.

In this note, we derive extensions to the results of \citet{Wasserman:2020aa}
for the context of composite likelihood-based (or equivalently, pseudo-likelihood-based)
inference, as considered in \citet{Lindsay1988}, \citet{Arnold1991},
\citet{Molenberghs2005}, \citet{Varin2011}, \citet{Yi:2014aa},
and \citet{Nguyen:2018ab}, among numerous other texts. This includes
results for batch inference as well as sequential inference.

We proceed as follows. In Section 2, we present the main results that
extend upon the theorems of \citet{Wasserman:2020aa}. Proofs are
then provided in Section 3. Technical requirements to prove our results
are provided in the Appendix.

\section{Main results}

Let $2^{\left[d\right]}$ be the power set of $\left[d\right]=\left\{ 1,\dots,d\right\} $,
and let $\mathbb{S}_{d}=2^{\left[d\right]}\backslash\left\{ \emptyset\right\} $.
For each $S\in\mathbb{S}_{d}$, let $S=\left\{ s_{1},\dots,s_{\left|S\right|}\right\} \subseteq\left[d\right]$,
where $\left|S\right|$ is the size of $S$. Further, let $\mathbb{T}_{d}$
be the set of all divisions of $\left[d\right]$ into two non-empty
subsets. For elements $T\in\mathbb{T}_{d}$, we write $\overleftarrow{T}=\left\{ \overleftarrow{t}_{1},\dots,\overleftarrow{t}_{\left|\overleftarrow{T}\right|}\right\} \subset\left[d\right]$
and $\overrightarrow{T}=\left\{ \overrightarrow{t}_{1},\dots,\overrightarrow{t}_{\left|\overrightarrow{T}\right|}\right\} \subset\left[d\right]\backslash\overleftarrow{T}$
to be the ``left-hand'' and ``right-hand'' subsets of the division
$T$, respectively. We note that $\left|\mathbb{S}_{d}\right|=2^{d}-1$
and $\left|\mathbb{T}_{d}\right|=3^{d}-2^{d+1}+1$.

For each $S$, let $\alpha_{S}\ge0$ and for each $T$, let $\beta_{T}\ge0$.
We shall call these coefficients weights. Put the weights $\alpha_{S}$
and $\beta_{T}$ in the vectors $\bm{\alpha}=\left(\alpha_{S}\right)_{S\in\mathbb{S}_{d}}$
and $\bm{\beta}=\left(\beta_{T}\right)_{T\in\mathbb{T}_{d}}$, respectively,
and assume that 
\begin{equation}
\gamma=\sum_{S\in\mathbb{S}_{d}}\alpha_{S}+\sum_{T\in\mathbb{T}_{d}}\beta_{T}>0\text{.}\label{eq: gamma}
\end{equation}

Given the set of weights $\bm{\alpha}$ and $\bm{\beta}$, we define
the individual composite likelihood (CL) for $\bm{X}$ as 
\[
p_{\bm{\alpha},\bm{\beta}}\left(\bm{X};\bm{\theta}\right)=\prod_{S\in\mathbb{S}_{d}}\left[p\left(\bm{X}_{S};\bm{\theta}\right)\right]^{\alpha_{S}/\gamma}\prod_{T\in\mathbb{T}_{d}}\left[p\left(\bm{X}_{\overleftarrow{T}}|\bm{X}_{\overrightarrow{T}};\bm{\theta}\right)\right]^{\beta_{T}/\gamma}\text{,}
\]
where $\bm{X}_{S}=\left(X_{s_{1}},\dots X_{s_{\left|S\right|}}\right)$,
$\bm{X}_{\overleftarrow{T}}=\left(X_{\overleftarrow{t}_{1}},\dots,X_{\overleftarrow{t}_{\left|\overleftarrow{T}\right|}}\right)$,
and $\bm{X}_{\overrightarrow{T}}=\left(X_{\overrightarrow{t}_{1}},\dots,X_{\overrightarrow{t}_{\left|\overrightarrow{T}\right|}}\right)$.
That is, $p\left(\bm{x}_{S};\bm{\theta}\right)$ is the marginal PDF
with respect to the coordinates of $\bm{X}$ corresponding to the
subset $S$, and $p\left(\bm{x}_{\overleftarrow{T}}|\bm{x}_{\overrightarrow{T}};\bm{\theta}\right)$
is the conditional PDF of the coordinates corresponding to $\overleftarrow{T}$,
conditioned on the coordinates corresponding to $\overrightarrow{T}$.

Assume, as in the introduction, that the elements of $\mathbf{X}_{2n}$
are sampled IID from a DGP with distribution $\mathcal{P}_{\bm{\theta}^{*}}$
and PDF $p\left(\bm{x};\bm{\theta}^{*}\right)$, for some $\bm{\theta}^{*}\in\Theta$.
Further, $\tilde{\bm{\theta}}_{n}^{k}$ are still generic estimators
of $\bm{\theta}^{*}$, for each $k\in\left\{ 0,1\right\} $. 

Let
\[
L_{\bm{\alpha},\bm{\beta}}\left(\bm{\theta};\mathbf{X}_{n}^{k}\right)=\prod_{i=1}^{n}p_{\bm{\alpha},\bm{\beta}}\left(\bm{X}_{i}^{k};\bm{\theta}\right)
\]
denote the composite likelihood of the subsample $\mathbf{X}_{n}^{k}$,
evaluated at $\bm{\theta}\in\Theta$. We shall write the split composite
likelihood ratio statistics (CLRSs) and the swapped CLRS as
\[
U_{\bm{\alpha},\bm{\beta},n}^{k}\left(\bm{\theta}\right)=\frac{L_{\bm{\alpha},\bm{\beta}}\left(\tilde{\bm{\theta}}_{n}^{1-k};\mathbf{X}_{n}^{k}\right)}{L_{\bm{\alpha},\bm{\beta}}\left(\bm{\theta};\mathbf{X}_{n}^{k}\right)}\text{,}
\]
for each $k\in\left\{ 0,1\right\} $, and 
\[
\bar{U}_{\bm{\alpha},\bm{\beta},n}\left(\bm{\theta}\right)=\frac{U_{\bm{\alpha},\bm{\beta},n}^{0}\left(\bm{\theta}\right)+U_{\bm{\alpha},\bm{\beta},n}^{1}\left(\bm{\theta}\right)}{2}\text{,}
\]
respectively. 

Let 
\[
C_{n}^{\alpha}=\left\{ \bm{\theta}\in\Theta:U_{\bm{\alpha},\bm{\beta},n}^{0}\left(\bm{\theta}\right)\le1/\alpha\right\} 
\]
and 
\[
\bar{C}_{n}^{\alpha}=\left\{ \bm{\theta}\in\Theta:\bar{U}_{\bm{\alpha},\bm{\beta},n}\left(\bm{\theta}\right)\le1/\alpha\right\} 
\]
be universal confidence set estimators. We are now ready to establish
our first result regarding finite-sample validity of $C_{n}^{\alpha}$
and $\bar{C}_{n}^{\alpha}$.
\begin{prop}
\label{prop CLRS}The confidence set estimators $C_{n}^{\alpha}$
and $\bar{C}_{n}^{\alpha}$ are finite sample valid $100\left(1-\alpha\right)\%$
confidence sets for $\bm{\theta}^{*}$. That is,
\[
\mathrm{Pr}_{\bm{\theta}^{*}}\left(\bm{\theta}^{*}\in C_{n}^{\alpha}\right)\ge1-\alpha\text{,}
\]
and 
\[
\mathrm{Pr}_{\bm{\theta}^{*}}\left(\bm{\theta}^{*}\in\bar{C}_{n}^{\alpha}\right)\ge1-\alpha\text{,}
\]
for every $n\in\mathbb{N}$.
\end{prop}
Consider the null and alternative hypotheses

\begin{equation}
\text{H}_{0}:\bm{\theta}\in\Theta_{0}\text{, and }\text{H}_{1}:\bm{\theta}\in\Theta\backslash\Theta_{0}\text{.}\label{eq: hypotheses}
\end{equation}
Due to the duality between confidence sets and hypothesis tests (cf.
Thm. 2.3 of \citealp[Appendix 1]{Hochberg:1987aa}), Proposition \ref{prop CLRS}
can be used to construct simple hypothesis tests using the rejection
rules: reject $\text{H}_{0}$ if $C_{n}^{\alpha}\cap\Theta_{0}=\emptyset$
or if $\bar{C}_{n}^{\alpha}\cap\Theta_{0}=\emptyset$. Both of these
tests control the Type I error at the correct level of significance
$\alpha$. However, these tests may be difficult to use when the shapes
of $\Theta_{0}$, $C_{n}^{\alpha}$, and $\bar{C}_{n}^{\alpha}$ are
complex and difficult to compute.

Let
\begin{equation}
\hat{\bm{\theta}}_{n}^{k}=\underset{\bm{\theta}\in\Theta_{0}}{\arg\max}\,L_{\bm{\alpha},\bm{\beta}}\left(\bm{\theta};\mathbf{X}_{n}^{k}\right)\label{eq: MCLE}
\end{equation}
denote the maximum CL estimator (MCLE) computed using the subset $\mathbf{X}_{n}^{k}$,
for each $k\in\left\{ 0,1\right\} $. Using the MCLEs, we can construct
tests that are more akin to the traditional likelihood ratio test
or the pseudo-likelihood ratio test of \citet{Molenberghs2005}. To
construct our tests, we require the split test statistics
\[
V_{\bm{\alpha},\bm{\beta},n}^{k}=\frac{L_{\bm{\alpha},\bm{\beta}}\left(\tilde{\bm{\theta}}_{n}^{1-k};\mathbf{X}_{n}^{k}\right)}{L_{\bm{\alpha},\bm{\beta}}\left(\hat{\bm{\theta}}_{n}^{k};\mathbf{X}_{n}^{k}\right)}\text{,}
\]
for each $k$, and the swapped test statistic
\[
\bar{V}_{\bm{\alpha},\bm{\beta},n}=\frac{V_{\bm{\alpha},\bm{\beta},n}^{0}+V_{\bm{\alpha},\bm{\beta},n}^{1}}{2}\text{.}
\]
We define the split composite likelihood ratio test (CLRT) and the
swapped CLRT via the rules: reject $\text{H}_{0}$ if $V_{\bm{\alpha},\bm{\beta},n}^{0}>1/\alpha$
or if $\bar{V}_{\bm{\alpha},\bm{\beta},n}>1/\alpha$, respectively.
The following result establishes the correctness of the split and
swapped CLRTs.
\begin{prop}
\label{prop CLRT}The split and the swapped CLRTs control the Type
I error at the level $\alpha$, for all $n\in\mathbb{N}$. That is,
\[
\sup_{\bm{\theta}^{*}\in\Theta_{0}}\mathrm{Pr}_{\bm{\theta}^{*}}\left(V_{\bm{\alpha},\bm{\beta},n}^{0}>1/\alpha\right)\le\alpha\text{,}
\]
and 
\[
\sup_{\bm{\theta}^{*}\in\Theta_{0}}\mathrm{Pr}_{\bm{\theta}^{*}}\left(\bar{V}_{\bm{\alpha},\bm{\beta},n}>1/\alpha\right)\le\alpha\text{.}
\]
\end{prop}

\subsection{Always-valid inference}

Instead of observing $\mathbf{X}_{n}=\mathbf{X}_{n}^{0}$ in a single
batch, we now consider that the IID elements of $\mathbf{X}_{n}$
(i.e., $\bm{X}_{1},\bm{X}_{2},\dots$) arrive sequentially, from distribution
$\mathcal{P}_{\bm{\theta}^{*}}$. For each $n\in\mathbb{N}$, we wish
to conduct a test of the hypotheses \ref{eq: hypotheses}.

Let $\tilde{\bm{\theta}}_{n-1}^{1}$ be a generic non-anticipating
estimator of $\bm{\theta}^{*}$ (i.e., $\tilde{\bm{\theta}}_{n-1}^{1}$
is only dependent on the data in $\mathbf{X}_{n-1}$), and let $\hat{\bm{\theta}}_{n}^{0}$
be the same as it was defined in (\ref{eq: MCLE}). Further, define
the running CLRT test statistic
\[
M_{\bm{\alpha},\bm{\beta},n}=\frac{\prod_{i=1}^{n}p_{\bm{\alpha},\bm{\beta}}\left(\bm{X}_{i};\tilde{\bm{\theta}}_{i-1}^{1}\right)}{\prod_{i=1}^{n}p_{\bm{\alpha},\bm{\beta}}\left(\bm{X}_{i};\hat{\bm{\theta}}_{n}^{0}\right)}
\]
and at any time $n$, reject $\text{H}_{0}$ and stop the sequence
of tests if $M_{\bm{\alpha},\bm{\beta},n}>1/\alpha$. If $\nu_{\bm{\theta}}$
denotes the time at which the test stops, under the rejection rule,
given that the data arrises IID from $\mathcal{P}_{\bm{\theta}}$,
then we establish the fact that $\nu_{\bm{\theta}^{*}}$ is finite
with probability at most $\alpha$.
\begin{prop}
\label{prop stopping time}The running CLRT has Type I error at most
$\alpha$. That is
\[
\sup_{\bm{\theta}^{*}\in\Theta_{0}}\mathrm{Pr}_{\bm{\theta}^{*}}\left(\nu_{\bm{\theta}^{*}}<\infty\right)\le\alpha\text{.}
\]
\end{prop}
Let $P_{n}=1/M_{\bm{\alpha},\bm{\beta},n}$ and $\tilde{P}_{n}=\min_{s\le n}\left(1/M_{\bm{\alpha},\bm{\beta},n}\right)$
be $p$-values for the test of (\ref{eq: hypotheses}) and let $N\in\mathbb{N}$
be a random variable. The following result establishes that both $P_{N}$
and $\tilde{P}_{N}$ are valid.
\begin{prop}
\label{prop p-value sequence}For any random $N$, not necessarily
a stopping time, $P_{N}$ and $\tilde{P}_{N}$ are valid $p$-values.
That is
\[
\sup_{\bm{\theta}^{*}\in\Theta_{0}}\mathrm{Pr}_{\bm{\theta}^{*}}\left(P_{N}\le\alpha\right)\le\alpha\text{,}
\]
and 
\[
\sup_{\bm{\theta}^{*}\in\Theta_{0}}\mathrm{Pr}_{\bm{\theta}^{*}}\left(\tilde{P}_{N}\le\alpha\right)\le\alpha\text{,}
\]
for all $\alpha\in\left[0,1\right]$.
\end{prop}
We define a confidence sequence for $\bm{\theta}^{*}$ as an infinite
sequence of confidence sets that are all simultaneously valid. In
the current context, such confidence sequence are $\left(D_{n}^{\alpha}\right)_{n\in\mathbb{N}}$
and $\left(\tilde{D}_{n}^{\alpha}\right)_{n\in\mathbb{N}}$, where
\[
D_{n}^{\alpha}=\left\{ \bm{\theta}\in\Theta:R_{\bm{\alpha},\bm{\beta},n}\left(\bm{\theta}\right)\le1/\alpha\right\} \text{,}
\]
$\tilde{D}_{n}^{\alpha}=\bigcap_{m\le n}D_{m}^{\alpha}$, and
\begin{equation}
R_{\bm{\alpha},\bm{\beta},n}\left(\bm{\theta}\right)=\frac{\prod_{i=1}^{n}p_{\bm{\alpha},\bm{\beta}}\left(\bm{X}_{i};\tilde{\bm{\theta}}_{i-1}^{1}\right)}{\prod_{i=1}^{n}p_{\bm{\alpha},\bm{\beta}}\left(\bm{X}_{i};\bm{\theta}\right)}\text{.}\label{eq: R Ratio}
\end{equation}
The following result establishes the validity of $\left(D_{n}^{\alpha}\right)_{n\in\mathbb{N}}$
and $\left(\tilde{D}_{n}^{\alpha}\right)_{n\in\mathbb{N}}$.
\begin{prop}
\label{prop confidence sequence}The confidence sequences $\left(D_{n}^{\alpha}\right)_{n\in\mathbb{N}}$
and $\left(\tilde{D}_{n}^{\alpha}\right)_{n\in\mathbb{N}}$ are valid.
That is
\[
\mathrm{Pr}_{\bm{\theta}^{*}}\left(\forall n\in\mathbb{N}:\bm{\theta}^{*}\in D_{n}^{\alpha}\right)\ge1-\alpha
\]
and 
\[
\mathrm{Pr}_{\bm{\theta}^{*}}\left(\forall n\in\mathbb{N}:\bm{\theta}^{*}\in\tilde{D}_{n}^{\alpha}\right)\ge1-\alpha\text{.}
\]
\end{prop}

\section{Proofs}

The following result provides the primary mechanism under which Propositions
\ref{prop CLRS} and \ref{prop CLRT} can be established, and is a
direct analog to (\ref{eq: Wasserman Main}) for CLs. 
\begin{lem}
\label{lem main}If\textbf{ }$\mathbf{X}_{2n}$ is an IID sample from
a DGP with distribution $\mathcal{P}_{\bm{\theta}^{*}}$ and PDF $f\left(\bm{x};\bm{\theta}^{*}\right)$,
then $U_{\bm{\alpha},\bm{\beta},n}^{k}\left(\bm{\theta}^{*}\right)$
has bounded expectation $\text{E}_{\bm{\theta}^{*}}\left[U_{\bm{\alpha},\bm{\beta},n}^{k}\left(\bm{\theta}^{*}\right)\right]\le1$,
for each $k\in\left\{ 0,1\right\} $ and for all $n\in\mathbb{N}$.
\end{lem}
\begin{proof}
We shall prove the $k=0$ case. Let $\mathbf{x}_{n}=\left(\bm{x}_{1},\dots,\bm{x}_{n}\right)$
and write

\begin{align*}
 & \text{E}_{\bm{\theta}^{*}}\left[U_{\bm{\alpha},\bm{\beta},n}^{0}\left(\bm{\theta}^{*}\right)|\mathbf{X}_{n}^{1}\right]\\
= & \int_{\mathbb{X}^{n}}\frac{L_{\bm{\alpha},\bm{\beta}}\left(\tilde{\bm{\theta}}_{n}^{1};\mathbf{x}_{n}\right)}{L_{\bm{\alpha},\bm{\beta}}\left(\bm{\theta}^{*};\mathbf{x}_{n}\right)}L\left(\bm{\theta}^{*};\mathbf{x}_{n}\right)\text{d}\mathbf{x}_{n}\\
= & \int_{\mathbb{X}^{n}}\frac{\prod_{i=1}^{n}p_{\bm{\alpha},\bm{\beta}}\left(\bm{x}_{i};\tilde{\bm{\theta}}_{n}^{1}\right)}{\prod_{i=1}^{n}p_{\bm{\alpha},\bm{\beta}}\left(\bm{x}_{i};\bm{\theta}^{*}\right)}\prod_{i=1}^{n}p\left(\bm{x}_{i};\bm{\theta}^{*}\right)\text{d}\mathbf{x}_{n}\\
= & \int_{\mathbb{X}^{n}}\prod_{i=1}^{n}\frac{p_{\bm{\alpha},\bm{\beta}}\left(\bm{x}_{i};\tilde{\bm{\theta}}_{n}^{1}\right)}{p_{\bm{\alpha},\bm{\beta}}\left(\bm{x}_{i};\bm{\theta}^{*}\right)}p\left(\bm{x}_{i};\bm{\theta}^{*}\right)\text{d}\mathbf{x}_{n}\text{.}
\end{align*}
Then, simplify the integrand by making the factorization
\begin{align*}
 & \frac{p_{\bm{\alpha},\bm{\beta}}\left(\bm{x}_{i};\tilde{\bm{\theta}}_{n}^{1}\right)}{p_{\bm{\alpha},\bm{\beta}}\left(\bm{x}_{i};\bm{\theta}^{*}\right)}p\left(\bm{x}_{i};\bm{\theta}^{*}\right)\\
\overset{\text{(i)}}{=} & \left(\frac{\prod_{S\in\mathbb{S}_{d}}\left[p\left(\bm{x}_{iS};\tilde{\bm{\theta}}_{n}^{1}\right)\right]^{\alpha_{S}/\gamma}\prod_{T\in\mathbb{T}_{d}}\left[p\left(\bm{x}_{i\overleftarrow{T}}|\bm{x}_{i\overrightarrow{T}};\tilde{\bm{\theta}}_{n}^{1}\right)\right]^{\beta_{T}/\gamma}}{\prod_{S\in\mathbb{S}_{d}}\left[p\left(\bm{x}_{iS};\bm{\theta}^{*}\right)\right]^{\alpha_{S}/\gamma}\prod_{T\in\mathbb{T}_{d}}\left[p\left(\bm{x}_{i\overleftarrow{T}}|\bm{x}_{i\overrightarrow{T}};\bm{\theta}^{*}\right)\right]^{\beta_{T}/\gamma}}\right)\\
 & \times\prod_{S\in\mathbb{S}_{d}}\left[p\left(\bm{x}_{i};\bm{\theta}^{*}\right)\right]^{\alpha_{S}/\gamma}\prod_{T\in\mathbb{T}_{d}}\left[p\left(\bm{x}_{i};\bm{\theta}^{*}\right)\right]^{\beta_{T}/\gamma}\\
\overset{\text{(ii)}}{=} & \prod_{S\in\mathbb{S}_{d}}\left[p\left(\bm{x}_{iS};\tilde{\bm{\theta}}_{n}^{1}\right)\right]^{\alpha_{S}/\gamma}\prod_{T\in\mathbb{T}_{d}}\left[p\left(\bm{x}_{i\overleftarrow{T}}|\bm{x}_{i\overrightarrow{T}};\tilde{\bm{\theta}}_{n}^{1}\right)\right]^{\beta_{T}/\gamma}\\
 & \times\prod_{S\in\mathbb{S}_{d}}\left[p\left(\bm{x}_{i,\left[d\right]\backslash S}|\bm{x}_{iS};\bm{\theta}^{*}\right)\right]^{\alpha_{S}/\gamma}\prod_{T\in\mathbb{T}_{d}}\left[p\left(\bm{x}_{i,\left[d\right]\backslash\left(\overleftarrow{T}\cup\overrightarrow{T}\right)}|\bm{x}_{i\overleftarrow{T}},\bm{x}_{i\overrightarrow{T}};\bm{\theta}^{*}\right)\right]^{\beta_{T}/\gamma}\\
 & \times\prod_{T\in\mathbb{T}_{d}}\left[p\left(\bm{x}_{i\overrightarrow{T}};\bm{\theta}^{*}\right)\right]^{\beta_{T}/\gamma}\\
\overset{\text{(iii)}}{=} & \prod_{S\in\mathbb{S}_{d}}\left[\tilde{p}\left(\bm{x}_{i};\tilde{\bm{\theta}}_{n}^{1},\bm{\theta}^{*}\right)\right]^{\alpha_{S}/\gamma}\prod_{T\in\mathbb{T}_{d}}\left[\check{p}\left(\bm{x}_{i};\tilde{\bm{\theta}}_{n}^{1},\bm{\theta}^{*}\right)\right]^{\beta_{T}/\gamma}
\end{align*}
where (i) is due to (\ref{eq: gamma}) and (ii) is due to the PDF
decompositions
\[
p\left(\bm{x}_{i};\bm{\theta}^{*}\right)=p\left(\bm{x}_{i,\left[d\right]\backslash S}|\bm{x}_{iS};\bm{\theta}^{*}\right)p\left(\bm{x}_{iS};\bm{\theta}^{*}\right)
\]
and 
\[
p\left(\bm{x}_{i};\bm{\theta}^{*}\right)=p\left(\bm{x}_{i,\left[d\right]\backslash\left(\overleftarrow{T}\cup\overrightarrow{T}\right)}|\bm{x}_{i\overleftarrow{T}},\bm{x}_{i\overrightarrow{T}};\bm{\theta}^{*}\right)p\left(\bm{x}_{i\overleftarrow{T}}|\bm{x}_{i\overrightarrow{T}};\bm{\theta}^{*}\right)p\left(\bm{x}_{i\overrightarrow{T}};\bm{\theta}^{*}\right)\text{.}
\]
The PDFs on line (iii) are then constructed as
\begin{equation}
\tilde{p}\left(\bm{x}_{i};\tilde{\bm{\theta}}_{n}^{1},\bm{\theta}^{*}\right)=p\left(\bm{x}_{i,\left[d\right]\backslash S}|\bm{x}_{iS};\bm{\theta}^{*}\right)p\left(\bm{x}_{iS};\tilde{\bm{\theta}}_{n}^{1}\right)\label{eq: pdf construct 1}
\end{equation}
and
\begin{equation}
\check{p}\left(\bm{x}_{i};\tilde{\bm{\theta}}_{n}^{1},\bm{\theta}^{*}\right)=p\left(\bm{x}_{i,\left[d\right]\backslash\left(\overleftarrow{T}\cup\overrightarrow{T}\right)}|\bm{x}_{i\overleftarrow{T}},\bm{x}_{i\overrightarrow{T}};\bm{\theta}^{*}\right)p\left(\bm{x}_{i\overleftarrow{T}}|\bm{x}_{i\overrightarrow{T}};\tilde{\bm{\theta}}_{n}^{1}\right)p\left(\bm{x}_{i\overrightarrow{T}};\bm{\theta}^{*}\right)\text{.}\label{eq: pdf construct 2}
\end{equation}
We then have 
\begin{align*}
 & \text{E}_{\bm{\theta}^{*}}\left[U_{\bm{\alpha},\bm{\beta},n}^{0}\left(\bm{\theta}^{*}\right)|\mathbf{X}_{n}^{1}\right]\\
= & \int_{\mathbb{X}^{n}}\prod_{i=1}^{n}\prod_{S\in\mathbb{S}_{d}}\left[\tilde{p}\left(\bm{x}_{i};\tilde{\bm{\theta}}_{n}^{1},\bm{\theta}^{*}\right)\right]^{\alpha_{S}/\gamma}\prod_{T\in\mathbb{T}_{d}}\left[\check{p}\left(\bm{x}_{i};\tilde{\bm{\theta}}_{n}^{1},\bm{\theta}^{*}\right)\right]^{\beta_{T}/\gamma}\text{d}\mathbf{x}_{n}\\
\overset{\text{(i)}}{=} & \prod_{i=1}^{n}\int_{\mathbb{X}}\prod_{S\in\mathbb{S}_{d}}\left[\tilde{p}\left(\bm{x}_{i};\tilde{\bm{\theta}}_{n}^{1},\bm{\theta}^{*}\right)\right]^{\alpha_{S}/\gamma}\prod_{T\in\mathbb{T}_{d}}\left[\check{p}\left(\bm{x}_{i};\tilde{\bm{\theta}}_{n}^{1},\bm{\theta}^{*}\right)\right]^{\beta_{T}/\gamma}\text{d}\bm{x}_{i}\\
\overset{\text{(ii)}}{\le} & \prod_{i=1}^{n}\prod_{S\in\mathbb{S}_{d}}\left[\int_{\mathbb{X}}\tilde{p}\left(\bm{x}_{i};\tilde{\bm{\theta}}_{n}^{1},\bm{\theta}^{*}\right)\text{d}\bm{x}_{i}\right]^{\alpha_{S}/\gamma}\prod_{T\in\mathbb{T}_{d}}\left[\int_{\mathbb{X}}\check{p}\left(\bm{x}_{i};\tilde{\bm{\theta}}_{n}^{1},\bm{\theta}^{*}\right)\text{d}\bm{x}_{i}\right]^{\beta_{S}/\gamma}\\
\overset{\text{(iii)}}{=} & \prod_{i=1}^{n}\prod_{S\in\mathbb{S}_{d}}1^{\alpha_{S}/\gamma}\prod_{T\in\mathbb{T}_{d}}1^{\beta_{S}/\gamma}=1\text{,}
\end{align*}
where (i) is due to separability, (ii) is due to the generalized H\"{o}lder's
inequality, and (iii) is due to the fact that (\ref{eq: pdf construct 1})
and (\ref{eq: pdf construct 2}) are PDFs. Finally, via the law of
iterated expectations, we have

\[
\text{E}_{\bm{\theta}^{*}}\left[U_{\bm{\alpha},\bm{\beta},n}^{0}\left(\bm{\theta}^{*}\right)\right]=\text{E}_{\bm{\theta}^{*}}\text{E}_{\bm{\theta}^{*}}\left[U_{\bm{\alpha},\bm{\beta},n}^{0}\left(\bm{\theta}^{*}\right)|\mathbf{X}_{n}^{1}\right]\le1\text{.}
\]
\end{proof}

\subsection{Proof of Proposition \ref{prop CLRS}}

We shall prove the fact that $\mathrm{Pr}_{\bm{\theta}^{*}}\left(\bm{\theta}^{*}\in\bar{C}_{n}^{\alpha}\right)\ge1-\alpha$
and not that the case for $C_{n}^{\alpha}$ can be proved in an identical
manner.

For any $\bm{\theta}^{*}\in\Theta$ and $n$, we have
\begin{align*}
 & \mathrm{Pr}_{\bm{\theta}^{*}}\left(\bm{\theta}^{*}\notin\bar{C}_{n}^{\alpha}\right)\\
= & \mathrm{Pr}_{\bm{\theta}^{*}}\left(\frac{U_{\bm{\alpha},\bm{\beta},n}^{0}\left(\bm{\theta}^{*}\right)+U_{\bm{\alpha},\bm{\beta},n}^{1}\left(\bm{\theta}^{*}\right)}{2}>1/\alpha\right)\\
\overset{\text{(i)}}{\le} & \alpha\text{E}_{\bm{\theta}^{*}}\left[\frac{U_{\bm{\alpha},\bm{\beta},n}^{0}\left(\bm{\theta}^{*}\right)+U_{\bm{\alpha},\bm{\beta},n}^{1}\left(\bm{\theta}^{*}\right)}{2}\right]\\
= & \frac{\alpha}{2}\text{E}_{\bm{\theta}^{*}}\left[U_{\bm{\alpha},\bm{\beta},n}^{0}\left(\bm{\theta}^{*}\right)\right]+\frac{\alpha}{2}\text{E}_{\bm{\theta}^{*}}\left[U_{\bm{\alpha},\bm{\beta},n}^{1}\left(\bm{\theta}^{*}\right)\right]\\
\overset{\text{(ii)}}{\le} & \frac{\alpha}{2}+\frac{\alpha}{2}=\alpha\text{,}
\end{align*}
where (i) is due to Markov's inequality and (ii) is due to Lemma \ref{lem main}.
We obtain the desired result by computing the complement
\[
\mathrm{Pr}_{\bm{\theta}^{*}}\left(\bm{\theta}^{*}\in\bar{C}_{n}\right)=1-\mathrm{Pr}_{\bm{\theta}^{*}}\left(\bm{\theta}^{*}\notin\bar{C}_{n}\right)\ge1-\alpha\text{.}
\]

\subsection{Proof of Proposition \ref{prop CLRT}}

We shall prove the result for the swapped CRLT and note that the split
CLRT result can be proved in an identical manner.

For any $\bm{\theta}^{*}\in\Theta_{0}$ and $n$, we have
\begin{align*}
 & \mathrm{Pr}_{\bm{\theta}^{*}}\left(\bar{V}_{\bm{\alpha},\bm{\beta},n}>1/\alpha\right)\\
= & \mathrm{Pr}_{\bm{\theta}^{*}}\left(\frac{V_{\bm{\alpha},\bm{\beta},n}^{0}+V_{\bm{\alpha},\bm{\beta},n}^{1}}{2}>\frac{1}{\alpha}\right)\\
\overset{\text{(i)}}{\le} & \alpha\text{E}_{\bm{\theta}^{*}}\left[\frac{V_{\bm{\alpha},\bm{\beta},n}^{0}+V_{\bm{\alpha},\bm{\beta},n}^{1}}{2}\right]\\
\overset{\text{(ii)}}{\le} & \alpha\text{E}_{\bm{\theta}^{*}}\left[\frac{U_{\bm{\alpha},\bm{\beta},n}^{0}\left(\bm{\theta}^{*}\right)+U_{\bm{\alpha},\bm{\beta},n}^{1}\left(\bm{\theta}^{*}\right)}{2}\right]\\
= & \frac{\alpha}{2}\text{E}_{\bm{\theta}^{*}}\left[U_{\bm{\alpha},\bm{\beta},n}^{0}\left(\bm{\theta}^{*}\right)\right]+\frac{\alpha}{2}\text{E}_{\bm{\theta}^{*}}\left[U_{\bm{\alpha},\bm{\beta},n}^{1}\left(\bm{\theta}^{*}\right)\right]\\
\overset{\text{(iii)}}{\le} & \frac{\alpha}{2}+\frac{\alpha}{2}=\alpha\text{,}
\end{align*}
where (i) is due to Markov's inequality, (ii) is due to (\ref{eq: MCLE})
(i.e., $L_{\bm{\alpha},\bm{\beta}}\left(\hat{\bm{\theta}}_{n}^{k};\mathbf{X}_{n}^{k}\right)\ge L_{\bm{\alpha},\bm{\beta}}\left(\bm{\theta}^{*};\mathbf{X}_{n}^{k}\right)$,
for all $\bm{\theta}^{*}\in\Theta_{0}$), and (iii) is due to Lemma
\ref{lem main}. The desired result is thus obtained.

\subsection{Proof of Proposition \ref{prop stopping time}}

Under $\text{H}_{0}$, observe that $M_{\bm{\alpha},\bm{\beta},n}\le M_{n}^{*}$,
where $M_{n}^{*}=R_{\bm{\alpha},\bm{\beta},n}\left(\bm{\theta}^{*}\right)$
is as defined in (\ref{eq: R Ratio}), since $L_{\bm{\alpha},\bm{\beta}}\left(\hat{\bm{\theta}}_{n};\mathbf{X}_{n}\right)\ge L_{\bm{\alpha},\bm{\beta}}\left(\bm{\theta}^{*};\mathbf{X}_{n}\right)$,
for $\bm{\theta}^{*}\in\Theta_{0}$. Let $\left(\mathcal{F}_{n}\right)_{n\in\mathbb{N}\cup\left\{ 0\right\} }$
be the natural filtration, where $\mathcal{F}_{n}=\sigma\left(\mathbf{X}_{n}\right)$.
Upon defining $M_{0}^{*}=1$, notice that
\begin{align*}
\text{E}_{\bm{\theta}^{*}}\left[M_{n}^{*}|\mathcal{F}_{n-1}\right] & =\text{E}_{\bm{\theta}^{*}}\left[\frac{\prod_{i=1}^{n-1}p_{\bm{\alpha},\bm{\beta}}\left(\bm{X}_{i};\tilde{\bm{\theta}}_{i-1}^{1}\right)}{\prod_{i=1}^{n-1}p_{\bm{\alpha},\bm{\beta}}\left(\bm{X}_{i};\bm{\theta}^{*}\right)}\frac{p_{\bm{\alpha},\bm{\beta}}\left(\bm{X}_{n};\tilde{\bm{\theta}}_{n-1}^{1}\right)}{p_{\bm{\alpha},\bm{\beta}}\left(\bm{X}_{n};\bm{\theta}^{*}\right)}|\mathcal{F}_{n-1}\right]\\
 & =\text{E}_{\bm{\theta}^{*}}\left[M_{n-1}^{*}\frac{p_{\bm{\alpha},\bm{\beta}}\left(\bm{X}_{n};\tilde{\bm{\theta}}_{n-1}^{1}\right)}{p_{\bm{\alpha},\bm{\beta}}\left(\bm{X}_{n};\bm{\theta}^{*}\right)}|\mathcal{F}_{n-1}\right]\\
 & =M_{n-1}^{*}\text{E}_{\bm{\theta}^{*}}\left[\frac{p_{\bm{\alpha},\bm{\beta}}\left(\bm{X}_{n};\tilde{\bm{\theta}}_{n-1}^{1}\right)}{p_{\bm{\alpha},\bm{\beta}}\left(\bm{X}_{n};\bm{\theta}^{*}\right)}|\mathcal{F}_{n-1}\right]\\
 & \overset{\text{(i)}}{\le}M_{n-1}^{*}\text{,}
\end{align*}
where (i) is established using the same argument as used in Lemma
\ref{lem main}. Thus, we have established that $\left(M_{n}^{*}\right)_{n\in\mathbb{N}\cup\left\{ 0\right\} }$
is a supermartingale, adapted to $\left(\mathcal{F}_{n}\right)_{n\in\mathbb{N}\cup\left\{ 0\right\} }$.
Upon application of Lemma \ref{lem Ville}, we have
\begin{equation}
\mathrm{Pr}_{\bm{\theta}^{*}}\left(\exists n\in\mathbb{N}:M_{n}^{*}\ge1/\alpha\right)\le\alpha M_{0}^{*}=\alpha\text{.}\label{eq: M martingale}
\end{equation}

Note that 
\[
\left\{ \nu_{\bm{\theta}^{*}}=\infty\right\} =\left\{ \forall n\in\mathbb{N}:M_{\bm{\alpha},\bm{\beta},n}<1/\alpha\right\} 
\]
and hence
\[
\left\{ \nu_{\bm{\theta}^{*}}<\infty\right\} =\left\{ \exists n\in\mathbb{N}:M_{\bm{\alpha},\bm{\beta},n}\ge1/\alpha\right\} \text{.}
\]
We obtain the desired result, since
\begin{align*}
\mathrm{Pr}_{\bm{\theta}^{*}}\left(\nu_{\bm{\theta}^{*}}<\infty\right) & =\mathrm{Pr}_{\bm{\theta}^{*}}\left(\exists n\in\mathbb{N}:M_{\bm{\alpha},\bm{\beta},n}\ge1/\alpha\right)\\
 & \le\mathrm{Pr}_{\bm{\theta}^{*}}\left(\exists n\in\mathbb{N}:M_{n}^{*}\ge1/\alpha\right)\\
 & \le\alpha\text{.}
\end{align*}

\subsection{Proof of Proposition \ref{prop p-value sequence}}

Firstly note that
\begin{align*}
\left\{ \exists n\in\mathbb{N}:M_{\bm{\alpha},\bm{\beta},n}\ge1/\alpha\right\}  & =\left\{ \exists n\in\mathbb{N}:P_{n}\le\alpha\right\} \\
 & =\bigcup_{n\in\mathbb{N}}\left\{ P_{n}\le\alpha\right\} \text{,}
\end{align*}
and apply Lemma \ref{lem equivalence of events} to establish the
validity of $P_{N}$.

In order to establish the validity of $\tilde{P}_{N}$, we note that
\[
\left\{ \tilde{P}_{n}\le\alpha\right\} =\bigcup_{m\le n}\left\{ P_{m}\le\alpha\right\} 
\]
and hence
\begin{align*}
\bigcup_{n\in\mathbb{N}}\left\{ \tilde{P}_{n}\le\alpha\right\}  & =\bigcup_{n\in\mathbb{N}}\bigcup_{m\le n}\left\{ P_{m}\le\alpha\right\} \\
 & =\bigcup_{n\in\mathbb{N}}\left\{ P_{n}\le\alpha\right\} \text{.}
\end{align*}

\subsection{Proof of Proposition \ref{prop confidence sequence}}

Notice that $R_{\bm{\alpha},\bm{\beta},n}\left(\bm{\theta}^{*}\right)=M_{n}^{*}$,
for each $n\in\mathbb{N}$, where $M_{n}^{*}$ is as defined in the
proof of Proposition \ref{prop stopping time}. Then
\begin{align*}
\mathrm{Pr}_{\bm{\theta}^{*}}\left(\exists n\in\mathbb{N}:\bm{\theta}^{*}\notin D_{n}^{\alpha}\right) & =\mathrm{Pr}_{\bm{\theta}^{*}}\left(\exists n\in\mathbb{N}:R_{\bm{\alpha},\bm{\beta},n}\left(\bm{\theta}^{*}\right)>1/\alpha\right)\\
 & \le\mathrm{Pr}_{\bm{\theta}^{*}}\left(\exists n\in\mathbb{N}:M_{n}^{*}\ge1/\alpha\right)\\
 & \le\alpha\text{,}
\end{align*}
due to (\ref{eq: M martingale}). Thus, we have demonstrated the validity
of $\left(D_{n}^{\alpha}\right)_{n\in\mathbb{N}}$.

To prove the validity of $\left(\tilde{D}_{n}^{\alpha}\right)_{n\in\mathbb{N}}$,
write
\begin{align*}
\left\{ \bm{\theta}^{*}\notin\tilde{D}_{n}^{\alpha}\right\}  & =\left\{ \bm{\theta}^{*}\notin\bigcap_{m\le n}D_{n}^{\alpha}\right\} \\
 & =\bigcup_{m\le n}\left\{ \bm{\theta}^{*}\notin D_{n}^{\alpha}\right\} \text{.}
\end{align*}
Thus
\begin{align*}
\left\{ \exists n\in\mathbb{N}:\bm{\theta}^{*}\notin\tilde{D}_{n}^{\alpha}\right\}  & =\bigcup_{n\in\mathbb{N}}\left\{ \bm{\theta}^{*}\notin\tilde{D}_{n}^{\alpha}\right\} \\
 & =\bigcup_{n\in\mathbb{N}}\bigcup_{m\le n}\left\{ \bm{\theta}^{*}\notin D_{n}^{\alpha}\right\} \\
 & =\bigcup_{n\in\mathbb{N}}\left\{ \bm{\theta}^{*}\notin D_{n}^{\alpha}\right\} \\
 & =\left\{ \exists n\in\mathbb{N}:\bm{\theta}^{*}\notin D_{n}^{\alpha}\right\} \text{,}
\end{align*}
as required.

\section*{Appendix}

\subsection*{Technical requirements}

We state some technical results that are required throughout the text.
References for unproved results are provided at the end of the section.
\begin{lem}
[Ville's Inequality]\label{lem Ville}If $\left(Y_{n}\right)_{n\in\mathbb{N}\cup\left\{ 0\right\} }$
is a non-negative supermartingale, adapted to the filtration $\left(\mathcal{F}_{n}\right)_{n\in\mathbb{N}\cup\left\{ 0\right\} }$.
Then, for any $\alpha>0$, we have
\[
\mathrm{Pr}\left(\exists n\in\mathbb{N}:Y_{n}\ge1/\alpha\right)\le\alpha Y_{0}\text{.}
\]
\end{lem}
\begin{lem}
\label{lem equivalence of events}Let $\left(\mathsf{A}_{n}\right)_{n\in\mathbb{N}}$
be a sequence of events in some filtered probability space, and let
$\mathsf{A}_{\infty}=\lim\sup_{n\rightarrow\infty}\mathsf{A}_{n}$.
If $\alpha\in\left[0,1\right]$, then the following statements are
equivalent: (a) $\mathrm{Pr}\left(\bigcup_{n=1}^{\infty}\mathsf{A}_{n}\right)\le\alpha$,
(b) $\mathrm{Pr}\left(\mathsf{A}_{N}\right)\le\alpha$ for all random
(potentially not stopping times) $N$, (c) $\mathrm{Pr}\left(\mathsf{A}_{\nu}\right)\le\alpha$
for all stopping times $\nu$ (possibly infinite).
\end{lem}
Lemma \ref{lem Ville} appears as Lemma 1 in \citet{Howard:2020aa}
(see also \citealp*[Lem. 1.1]{Stout:1973aa}). Lemma \ref{lem equivalence of events}
appears as Lemma 3 in \citet{Howard:2020ab}.

\bibliographystyle{apalike2}
\bibliography{20200901_MASTERBIB}

\end{document}